\documentclass[12pt,a4paper]{article}

\usepackage{amssymb, amsmath, amsthm}

\usepackage[cm]{fullpage}
\usepackage[english]{babel}
\usepackage[cp1250]{inputenc}
\usepackage[T1]{fontenc}
\usepackage[pdftex]{graphicx}
\usepackage{booktabs}

\newtheorem{df}{Definition}
\newtheorem{thm}{Theorem}

\newtheorem{prop}{Proposition}

\newtheorem{cor}{Corollary} 

\title{Asymptotic analysis of internal relaxation-oscillations in a conceptual climate model}
\author{\L ukasz P\l ociniczak\thanks{Faculty of Pure and Applied Mathematics, Wroc{\l}aw University of Science and Technology, Wyb. Wyspia{\'n}skiego 27, 50-370 Wroc{\l}aw, Poland}$\;^,$\footnote{Email: lukasz.plociniczak@pwr.edu.pl}}
\date{}

\begin{document}
\maketitle
	
\begin{abstract}	
	We construct a dynamical system based on the KCG (K\"all\'en, Crafoord, Ghil) conceptual climate model which includes the ice-albedo and precipitation-temperature feedbacks. Further, we classify the stability of various critical points of the system and identify a parameter which change generates a Hopf bifurcation. This gives rise to a stable limit cycle around a physically interesting critical point. Moreover, it follows from the general theory that the periodic orbit exhibits relaxation-oscillations which are a characteristic feature of the Pleistocene ice-ages. We provide an asymptotic analysis of their behaviour and derive a formula for the period along with several estimates. They, in turn, are in a decent agreement with paleoclimatic data and are independent of any parametrization used. Whence, our simple but robust model shows that a climate may exhibit internal relaxation-oscillations without any external forcing and for a wide range of parameters.
	\\
		
\end{abstract}

\section{Introduction}
Conceptual climate models provide a feasible mean of grasping the most important mechanisms of the complex system as the climate itself \cite{Sal02,Cru12}. Focusing on only the most essential features they can help to understand the basic laws that govern the dynamics of various parameters of the climate. Of course, by construction, they are unable to predict the temperature changes in a great detail (as opposed to General Climate Models (GCMs)) but still are invaluable tool in understanding the Earth system evolution \cite{McG05}. Their virtue is that they usually are low-dimensional dynamical systems that can be understood by analytical or numerical means without the use of the supercomputers. The analytical approach has the advantage of being able to simultaneously tackle an infinite number of initial conditions providing results that can then be verified with the use of GCMs.

One of the remarkable features of the paleoclimatic record is the emergence of relaxation-oscillations of the climate \cite{Lis05,Cru12}. These asymmetrical variations of the ice sheet extent indicate that there should exist a nonlinear mechanism based on various feedbacks that can drive such dynamics. More specifically, a slow growth of the ice sheet is followed by a rapid deglaciation \cite{Lis05}. Conceptual models that try to explain this phenomena are reviewed in \cite{Cru12}. 

In the literature there is a number of approaches to conceptual climate modelling with varying complexity and descriptive variables used. This work is influenced by K\"all\'en, Crafoord and Ghil's model that analyses temperature and ice sheet evolution which essentially are the energy and mass balances \cite{Kal97,Ghi81,Ghi83}. The beginnings of formulating these balances as differential equations can be traced back to Budyko \cite{Bud69}, Sellers \cite{Sel69} and Weertman \cite{Wee76}. The Budyko-Sellers model show that there is a hysteric response of the climate to the solar forcing which indicates that the climate can abruptly enter a completely remote cold or warm state (a tipping point). On the other hand, one of the main features of the KCG model indicates that the climate can oscillate without any external (astronomical) forcing as happens in the Milankovitch theory \cite{Mil98,Ber88,Hay76}. As a matter of fact, the periodic change in Earth's axis tilt, precession and eccentricity proposed by Milankovitch to describe glacial episodes are nowadays treated as a pacemaker of the nonlinear oscillator being the climate itself \cite{De13,Tzi06}. Similar considerations lead Fowler and his collaborators to incorporate the energy and mass balances into the chemical reaction framework following CO$_2$ evolution \cite{Fow13,Fow15}. As authors show, the relaxation-oscillations of the climate can be realised as an inclusion of a mechanism responsible for formation of a progracial lakes.

Another conceptual climate framework was developed by Saltzman and collaborators \cite{Maa90}. They focus on describing the dynamics of total ice volume, CO$_2$ concentration and deep-ocean temperature. This model has been investigated for many years and the whole theory is summarized in \cite{Sal02} (but see also \cite{Eng17}). The important feature of that model is the emergence of relaxation-oscillations which describe the Pleistocene ice ages. 

More mathematically inclined considerations on conceptual climate models have been conducted by McGehee and collaborators. In \cite{McG12,Wal13} they have investigated the full version of the Budyko-Sellers model stated as a partial differential equation. If its steady state is approximated by quadratic functions the whole dynamics can be captured into a low-dimensional dynamical system analysed in \cite{McG14}. The model is constructed with a discontinuous vector field and possesses periodic orbits \cite{Wal16}.

In this work, which is a continuation of \cite{Plo18}, we investigate a generalization of the KCG model and show that for a reasonable parameter regime it predicts that the climate will undergo relaxation-oscillations. First, we state the model in its full generality and classify all its critical points according to the relative slopes of the nullclines. This is a reasonable geometrical approach since the model possesses a multitude of parameters which specify the position of the critical points. Their exact position is not necessarily needed to know as opposed to their stability or bifurcations. After investigating the general model we approximate it by focusing only on the physically most interesting critical point. It appears that for almost the whole range of the bifurcation parameter the system undergoes a relaxation-oscillation without any external forcing. With the use of matched asymptotic expansions we find the asymptotic formula for its period and amplitude. Further, we show that the former can be bounded by some elementary functions that produce very reasonable estimates on the glacial ages. The lower estimate appears to be completely independent from the particular choices of the functional forms of the albedo and accumulation/ablation. This makes it a universal bound. We close the paper with a numerical verification of our results and find that the asymptotic formula is decently accurate even for small values of the parameter. 

\section{Model statement and general results}
\subsection{Main dynamical system}
In \cite{Plo18} we have provided a thorough derivation of the generalized version of the KCG model firstly proposed in \cite{Kal97}. Here, we will just summarize the main features of the model and an interested Reader is invited to learn about all the details in the original works.

The main constituents of the KCG model are the energy and mass balances. The energy equation is of Budyko-Sellers type being a competition between incoming short-wave and outgoing long-wave radiation (OLR). Earth acquires its energy from the Sun in means of the high quality light which some part is reflected and some absorbed by the planet. The ratio of the reflected to incident energy is called the \textit{albedo} $\alpha$ and its present average value for Earth is around $0.3$ (see \cite{Fow11}). Moreover, within the naturally occurring features of the surface of our planet, the fresh snow has one of the highest values of $\alpha$ being $0.8-0.9$ while open ocean one of the lowest equalling $0.06$. Following KCG idea we will decompose albedo into two parts one representing the continents $\alpha_c$ and the other the oceans $\alpha_o$. The continental part describes the reflectivity of the land on which the ice sheet can form and advance. The oceanic albedo, in turn, depends mainly on the temperature since for cold climates the sea ice forms increasing the amount of light reflected. The variability of the albedo with respect to the amount of ice present is called the \textit{ice-albedo} feedback.  

The second ingredient of the model is the mass balance written as to describe the advances and retreats of the north hemisphere's ice sheet. This has firstly been proposed in \cite{Wee76}. Assuming zonal, i.e. longitudinal, symmetry the ice sheet's northern margin is in the Arctic ocean and the southward flow of the ice is taken to be perfectly plastic. The mass can change twofold: either by nourishment with a snowfall or by melting due to approach into the ablation zone at low latitudes. The imaginary boundary between these two regimes is called the \textit{snow line}. It essentially is the $0^\circ$ C isotherm and it can be visualized as a slanted line which height over the surface increases southward (meridionally). That is, high altitude near the equator has a comparable temperature to the low altitudes near the North pole. The accumulation and ablation regions are thus demarcated by the intersection of the snow line and the surface of the ice sheet. The connection with the energy equation comes from the fact that the rate of the snowfall depends on the temperature. Higher temperature increases evaporation of the oceans and hence increases the amount of water vapour in the atmosphere. Some of it can fall as a precipitation at higher latitudes as a snow building up the ice sheet. This mechanism is called the \textit{precipitation-temperature} feedback. 

For the neatness of the notation we will write the model in the nondimensional form the very beginning and describe what scalings have been used. Let $\theta$ be the temperature, $\lambda$ the southward ice sheet extent and $\tau$ denote the time. They are related to the dimensional values by
\begin{equation}
	\theta = \frac{4B}{Q} T, \quad \lambda = \frac{s^2}{H^2} L, \quad \tau = \frac{2}{3}\frac{ms}{H^2} t \quad \text{where} \quad H := \sqrt{\frac{4\tau_0}{3 \rho_i g}},
\end{equation}
and all appearing constants are explained in Tab. \ref{tab:Symbols}. Moreover, define the following parameters
\begin{equation}
	 \beta := -\frac{4A}{Q}, \quad \mu := \frac{3}{2} \frac{BH^2}{msc}, \quad \kappa := \frac{s h_0}{H^2}.
\end{equation}
Note that $\kappa$ can be negative and may depend on the temperature as an nondecreasing function. This is due to the fact that the height of the snow line over the Arctic Ocean $h_0$ changes according to the temperature. However, we will not specify the concrete form of $\kappa$ (in \cite{Ghi83} it is taken to be a linear function). The nondimensionalized model can now be written as follows 
\begin{equation}
\label{eqn:MainModel}
	\left\{
	\begin{array}{l}
		\dfrac{d\theta}{d\tau} = \mu \left(1-\beta-\gamma\alpha_c(\lambda)-(1-\gamma)\alpha_o(\theta)-\theta\right), \vspace{6pt}\\
		\dfrac{d\lambda}{d\tau} = \sqrt{\lambda} \left(\left(1+\xi(\theta)\right)\lambda_0(\lambda,\theta)-1\right),
	\end{array}	
	\right.
	\quad \text{for} \quad \theta,\lambda > 0,
\end{equation}
where the boundary between accumulation and ablation zones is given by
\begin{equation}
\label{eqn:IceLine}
	\lambda_0(\lambda,\theta) := \frac{1}{\lambda}\left(-\left(\kappa(\theta)+\lambda+\frac{1}{2}\right)+\sqrt{\kappa(\theta)+2\lambda+\frac{1}{4}}\right).
\end{equation}
Since $\alpha_c(\lambda)$, $\alpha_o(\theta)$ and $\xi(\theta)$ are monotone and bounded \cite{Kal97} it is useful to define the following well-known class of functions that can be used to represent those climatic quantities. 
\begin{df}
	A function $\sigma\colon\mathbb{R}\rightarrow\mathbb{R}$ is called \textbf{sigmoid} if it is bounded and differentiable with a non-negative derivative. As a normalization one can take $\lim\limits_{x\rightarrow\pm\infty} \sigma(x) = \pm 1$.
\end{df}  
Then, it is useful to take
\begin{equation}
\alpha_c(\lambda) := \frac{1}{2}\left(\alpha_0+\alpha_1 + (\alpha_0-\alpha_1)\;\sigma_{c}\left(\frac{\lambda-\lambda_\alpha}{\Delta \lambda}\right)\right), \quad \alpha_0 \leq \alpha_1,
\label{eqn:alphac}
\end{equation}
where $0\leq\alpha_{0,1}\leq 1$ are limits of $\alpha_c(\lambda)$ for $\theta\rightarrow \pm\infty$, $\lambda_\alpha$ is the translation and $\Delta \lambda$ being the steepness parameter. Similarly, we define
\begin{equation}
\alpha_o(\theta) := \frac{1}{2}\left(\alpha_+ +\alpha_- + (\alpha_+-\alpha_-)\;\sigma_{o}\left(\frac{\theta-\theta_\alpha}{\Delta \alpha}\right)\right), \quad \alpha_- \geq \alpha_+,
\label{eqn:alphao}
\end{equation}
and 
\begin{equation}
\xi(\theta) := \frac{1}{2}\left(\xi_+ +\xi_- + (\xi_+-\xi_-)\;\sigma_{\xi}\left(\frac{\theta-\theta_\xi}{\Delta \xi}\right)\right), \quad \xi_- \leq \xi_+.
\label{eqn:xi}
\end{equation}
Notice that $\alpha_c$ and $\xi$ are increasing while $\alpha_o$ is a decreasing function. These parametrizations of the climatic features are somewhat arbitrary since the physical processes that govern them are extremely complex and impossible to fully implement at the level of a conceptual model. As we will see in the sequel, many features of the relaxation-oscillations are independent of the particular choice of the sigmoid function. This makes the model more robust. Moreover, the various parameters can be tweaked in many ways to represent different climatic scenarios. It is not our aim to exactly describe reality but to show how aforementioned mechanisms give rise to an interesting dynamical behaviour which is also seen in the real world. Our choice of parameters is thus illustrative with a strong connection to the climate. 

\begin{table}
	\centering
	\begin{tabular}{rlc}
		\toprule
		Symbol & Meaning & Typical value \\
		\midrule
		$T$ & Globally averaged temperature & $-$ \\
		$L$ & Southward ice sheet extent & $-$ \\
		$t$, $\tau$ & Time variables & $-$ \\
		$Q$ & Solar constant & 1361 W m$^{-2}$ \\
		$\gamma$ & Continent to ocean area ratio & $0.3$ \\
		$A$ & Budyko constant in OLR flux & -267.96 W m$^{-2}$\\
		$B$ & Budyko constant in OLR flux & 1.74 W m$^{-2}$ K$^{-1}$\\
		$\tau_0$ & Ice sheet yield stress & 0.3$\times 10^5$ Pa \\
		$\rho_i$ & Ice density & 0.92$\times 10^3$ kg m$^{-3}$ \\
		$H$ & Ice sheet height scale & 2.1 m$^\frac{1}{2}$ \\
		$s$ & Snow line slope & 0.4 $\times 10^{-3}$\\
		$m$ & Ablation rate & $0.5$ m year$^{-1}$ \\
		$c$ & Atmosphere thermal capacity & $10^7$ J m$^{-2}$ K$^{-1}$ \\
		$h_0$ & Height of the snow line over Arctic ocean & $1.2\times 10^3$ m\\
		$\kappa$ & Nondimensional height of the snow line over Arctic ocean & 0.1 \\
		$\alpha_c$ & Continental albedo & $-$ \\
		$\alpha_o$ & Oceanic albedo & $-$ \\
		$\xi$ & Ratio of accumulation to ablation & $-$ \\
		$\lambda_0$ & The boundary between accumulation and ablation zones & $-$ \\
		$\alpha_1$, $\alpha_2$ & Parameters of the continental albedo & 0.25 and 1, respectively \\
		$\alpha_-$, $\alpha_+$ & limits of oceanic albedo & 0.6 and 0.22, respectively\\
		$\xi_-$, $\xi_+$ & limits of the ratio of accumulation and ablation & 0.1 and 0.5, respectively\\
		$\theta_\alpha$, $\Delta \alpha$ & translation and steepness parameters for oceanic albedo & 1.4 and 0.15, respectively \\
		$\theta_\xi$, $\Delta \xi$ & translation and steepness parameters for $\xi$ & 1.39 and 0.025, respectively \\
		$\sigma$ & An arbitrary sigmoid function & $-$ \\
		\bottomrule
	\end{tabular}
	\caption{Symbols used in the paper. The typical values are based on \cite{Kal97,Fow11,Fow13}.}
	\label{tab:Symbols}
\end{table}

Lastly, we would like to remark about the physical validity of the model (\ref{eqn:MainModel}). The conditions for these are the following
\begin{equation}
\label{eqn:ModelAssumptions}
\begin{split}
	\lambda_0(\lambda) \geq 0 \quad \text{(ice sheet not stagnant)}, \\
	\lambda \leq 1 \quad \text{(maximal ice sheet size)}, \\
	2\lambda + \kappa \geq 0 \quad \text{for} \quad \kappa < 0 \quad \text{(ice sheet sufficiently large)}.
\end{split}
\end{equation}
The physical derivation of the above is given in \cite{Plo18}. In the first and third of the above cases it is possible to provide different equations that build a meaningful model in these situations.

\subsection{The phase plane}
Here, we will provide a characterization of the dynamical behaviour of (\ref{eqn:MainModel}). First, denote the vector fields in (\ref{eqn:MainModel}) as
\begin{equation}
\label{eqn:Derivatives}
	F(\theta,\lambda) = \mu \left(1-\beta-\gamma\alpha_c(\lambda)-(1-\gamma)\alpha_o(\theta)-\theta\right), \quad G(\theta,\lambda) = \sqrt{\lambda} \left(\left(1+\xi(\theta)\right)\lambda_0(\lambda)-1\right).
\end{equation}
Computing the derivatives we have
\begin{equation}
\begin{split}
	\frac{\partial F}{\partial \theta} &= -\mu\left((1-\gamma)\frac{d\alpha_o}{d\theta}+1\right), \quad \frac{\partial F}{d\lambda} = -\mu\gamma \frac{d\alpha_c}{d \lambda}, \\
	\frac{\partial G}{\partial \theta} &= \sqrt{\lambda}\left(\frac{d\xi}{d\theta}\lambda_0 + (1+\xi)\frac{\partial\lambda_0}{\partial \theta}\right), \quad \frac{\partial G}{\partial \lambda} = \frac{1}{2\sqrt{\lambda}}\left(\left(1+\xi\right)\lambda_0-1\right)+\sqrt{\lambda} (1+\xi)\frac{\partial \lambda_0}{\partial\lambda},
\end{split}
\end{equation}
where we have suppressed the explicit writing of arguments for the brevity of notation. Further, the derivatives of the snow line $\lambda_0$ have the form
\begin{equation}
	\frac{\partial \lambda_0}{\partial \theta} = -\frac{1}{\lambda} \frac{d\kappa}{d\theta} \left(-1+\frac{1}{2\sqrt{\kappa+2\lambda+\frac{1}{4}}}\right), \quad \frac{\partial \lambda_0}{\partial \lambda} = \frac{1}{\lambda^2} \left(\kappa+\frac{1}{2}+\frac{\kappa+\lambda+\frac{1}{4}}{\sqrt{\kappa+2\lambda+\frac{1}{4}}}\right).
\end{equation}
Note that thanks to the third assumption in (\ref{eqn:ModelAssumptions}) the function $\lambda_0$ is $\theta$-nonincreasing. On the other hand, by the analysis of $d\lambda_0/d\lambda$ we see that for fixed $\kappa$ the function $\lambda_0(\cdot,\kappa)$ has a global maximum for $\lambda=\kappa(1+4\kappa)/2$ when $\kappa \geq 0$ and $\lambda = -\kappa/2$ when $\kappa<0$. 

The nullclines $h$ and $k$ of $\theta$ and $\lambda$ respectively, can be readily computed in a closed form giving
\begin{equation}
\label{eqn:NullclinesFull}
\begin{split}
	h(\theta) &= \alpha_c^{-1}\left(-\frac{1}{\gamma}\left((1-\gamma)\alpha_o(\theta)+\theta+\beta-1\right)\right), \\
	k^\pm(\theta) &= \frac{1}{2}\frac{(1+\xi)\xi}{(2+\xi)^2} \left(1-2\left(1+\frac{2}{\xi}\right)\kappa\pm\sqrt{1-4\left(1+\frac{2}{\xi}\right)\kappa}\right).
\end{split}
\end{equation}
The continental albedo is a monotone function of $\theta$ hence it has a well-defined inverse. Moreover, the $\lambda$-nullcline is a solution of the quadratic equation 
\begin{equation}
\label{eqn:NullclineCondition}
	\lambda_0 = \frac{1}{1+\xi},
\end{equation}
hence it has two branches denoted by $\pm$. They join at the common point for $\lambda = \kappa(1+4\kappa)/2$ when $\kappa\geq 0$ or $\lambda=-\kappa/2$ when $\kappa<0$. At that point $k^\pm$ has a singular derivative. Obviously, $k^+ \geq k^-$ and the sufficeint and necessary condition for they to exist is
\begin{equation}
\label{eqn:KappaAssumption}
\kappa \leq \frac{1}{4}\frac{\xi}{2+\xi}.
\end{equation}
since then the square root in (\ref{eqn:NullclinesFull}) is real. Inverting the above yields an important result that no ice sheet can exist for sufficiently high temperatures \cite{Wee76}.

The shape of the $\lambda$-nullcline depends on the particular form of $\kappa$. In \cite{Plo18} we have shown that $k^-(\theta) = O(\kappa^2)$ when $\kappa\rightarrow 0$ hence it becomes physically irrelevant as soon as the snow line is low enough. On the other hand, the other branch $k^+$ converges in that case to $(1+\xi)\xi/(2+\xi)^2$ which is an increasing function. This is physically the most important branch at which the critical points can be situated. We will come back to this issue in the next section. 

Since $\alpha_c^{-1}$ is increasing, the overall shape of the $\theta$-nullcline is not distorted by a composition with it. It is easy to observe that since $\alpha_o$ is a sigmoid function it is almost constant for arguments far to the left or far to the right of $\theta_\alpha$. In these regions, the dominant contribution comes from the decreasing linear function in the definition of $h$ in (\ref{eqn:NullclinesFull}). On the other hand, in the vicinity of $\theta_\alpha$ the sigmoid dominates making the whole graph of $h$ to resemble the S-shape. More precisely, computing the derivative we have
\begin{equation}
\label{eqn:NullclineDerivative}
	\frac{dh}{d\theta} = -\frac{1}{\gamma\alpha_c'}\left((1-\gamma)\frac{d\alpha_o}{d\theta}+1\right).
\end{equation}
Wee see that the above has two zeros only if the equation $(1-\gamma)\alpha_o'=-1$ has two solutions. Since the derivative of a sigmoid function has a single maximum which height is determined by the steepness parameter $\Delta\alpha$, the nullcline $h$ has two extrema for sufficiently small $\Delta\alpha$. Therefore, we further assume that
\begin{equation}
\label{eqn:AlphaAssumption}
	\Delta\alpha \text{ small enough for } \frac{dh}{d\theta}=0 \text{ to have two solutions } \theta_-<\theta_+.
\end{equation}
Whence, $h$ increases on $(\theta_-,\theta_+)$ and decreases otherwise. Now, we are ready to state the stability result.

\begin{thm}
Assume (\ref{eqn:KappaAssumption}) and (\ref{eqn:AlphaAssumption}). If $(\theta_c,\lambda_c)$ is a fixed critical point of the system (\ref{eqn:MainModel}) then the following holds. 
\begin{itemize}
	\item If $(\theta_c,\lambda_c)$ lies on $k^-$ then it is unstable.
	\item If $(\theta_c,\lambda_c)$ lies on $k^+$ then
	\begin{itemize}
		\item if $h$ is decreasing at that point then it is stable iff $dh/d\theta < dk^+/d\theta$,
		\item if $h$ is increasing at that point then it becomes unstable when $\mu$ increases through $\mu_c$. Moreover, if $dh/d\theta > dk^+/d\theta$ the Hopf bifrucation occurs. Here,
		\begin{equation}
			\mu_c = -\sqrt{\lambda} \left(1+\xi\right) \dfrac{d \lambda_0}{\partial \lambda} \left((1-\gamma)\dfrac{d\alpha_o}{d\theta}+1\right)^{-1} \quad \text{at} \quad (\theta_c,\lambda_c).
		\end{equation}
	\end{itemize}
\end{itemize}
\end{thm}
\begin{proof}
The stability of a particular critical point is determined from the eigenvalues $r_\pm$ of the Jacobi Matrix of the partial derivatives of $F$ and $G$ evaluated at the critical point $(\theta_c,\lambda_c)$. Denote this matrix by $J$ and then by elementary formulas we have
\begin{equation}
\label{eqn:Eigenvalues}
	r_\pm = \frac{1}{2}\left(\text{tr} J \pm \sqrt{\left(\text{tr} J\right)^2 -4\text{det} J}\right).
\end{equation}
Computing the trace and determinant at the \textit{critical point} we obtain
\begin{equation}
\label{eqn:TrDet}
\begin{split}
	\text{tr}J (\mu) &= -\mu \left((1-\gamma)\dfrac{d\alpha_o}{d\theta}+1\right) + \sqrt{\lambda} \left(1+\xi\right) \dfrac{d \lambda_0}{\partial \lambda}, \vspace{12 pt}\\
	\text{det} J (\mu)&=  \mu \gamma \sqrt{\lambda} \left(1+\xi\right)\frac{d\alpha_c}{\partial \lambda}  \dfrac{\partial \lambda_0}{\partial \lambda} \left(\dfrac{dh}{d\theta}-\dfrac{dk^\pm}{d\theta}\right)
\end{split}
\quad \text{at} \quad  (\theta,\lambda)=(\theta_c,\lambda_c),
\end{equation}
where we have used the Inverse Function Theorem to note that the derivatives of the nullclines are given by
\begin{equation}
\begin{split}
	\frac{dh}{d\theta} &= - \frac{\partial F}{\partial \theta} \left(\frac{\partial F}{\partial \lambda}\right)^{-1} = -\frac{1}{\gamma\alpha_c'}\left((1-\gamma)\frac{d\alpha_o}{d\theta}+1\right), \\
	\frac{dk^\pm}{d\theta} &= - \frac{\partial G}{\partial \theta} \left(\frac{\partial G}{\partial \lambda}\right)^{-1} = -\frac{1}{\lambda_0} \frac{\lambda_0^2\frac{dk}{d\theta}+\frac{\partial\lambda_0}{\partial \theta}}{(1+\xi)\frac{\partial \lambda_0}{\partial \lambda}}
\end{split}
\quad \text{at} \quad  (\theta,\lambda)=(\theta_c,\lambda_c).
\end{equation}
Notice that since $\lambda_0(\cdot,\kappa)$ has a maximum for fixed $\kappa$ and $\lambda$ on the upper branch of the nullcline, we have $\partial\lambda_0/\partial\lambda <0$ on $k^+$ and vice-versa.  

Assume now that the critical point lies on the upper branch $k^+$. If $h$ is decreasing at that point, we have $\partial F/\partial \theta < 0$ and $\partial\lambda_0/\partial\lambda <0$. Therefore, $\text{tr} J < 0$ for every $\mu>0$. On the other hand, the square root in the determinant is pure real or imaginary according to the sign of the following quadratic
\begin{equation}
	\left(\text{tr} J (\mu)\right)^2 -4\text{det} J(\mu).
\end{equation}
Although the exact zeros of the above can be readily computed, they are not needed for the stability result. For if $\mu$ is chosen to yield a real square root we have $r_- < 0$ and 
\begin{equation}
	r_+ = \frac{2\text{det} J}{\text{tr} J - \sqrt{\left(\text{tr} J\right)^2-4 \text{det} J}},
\end{equation}
which due to (\ref{eqn:Derivatives}) and (\ref{eqn:TrDet}) has its sign governed by the relative slope of the nullclines at the critical point. Therefore, by the Hartman-Grobman's Theorem the critical point is stable if $\det J >0$ which is $dh/d\theta < dk^+/d\theta$ and we have a node. Moreover, if the slope of $h$ is larger than this of $k^+$ the critical point is a saddle. On the other hand, in the case when the square root in $\det J(\mu)$ is imaginary, the critical point is always stable because $\text{Re}\, r_\pm = \text{tr} J (\mu)/ 2 < 0$ (a focus).

Next, assume that at the critical point $h$ is increasing and consider the upper branch of $\lambda$-nullcline. Now, since $\partial F/\partial \theta >0$ the trace is positive for $\mu>\mu_c$. The reasoning essentially the same as above applies to that case yielding a change of stability at $\mu_c$. When $dh/d\theta > dk^+/d\theta$ at the critical point we have $\text{tr} J(\mu_c) = 0$ and the square root in (\ref{eqn:Eigenvalues}) is purely imaginary. Therefore, as $\mu$ increases through $\mu_c$ the eigenvalues cross the imaginary axis on the complex plane with a non-zero speed because $\text{tr} J(\mu)$ is a linear function. The Hopf Bifurcation Theorem \cite{Per13} yields the result. 

Completely analogous reasoning can be applied to the lower branch $k^-$ and we omit the details. The important thing to keep in mind is that we cannot have $dh/d\theta > dk^-/d\theta$ at the critical point when $h$ is decreasing. For then the nullclines cannot intersect and there is no critical point. A careful sign counting helps to see that the lower branch is unstable. The proof is now complete. 
\end{proof}

We have thus seen that since the lower branch of the $\lambda$-nullcline is completely unstable, the physically meaningful phenomena can happen only on $k^+$. Now, we will elaborate on this particular case.

\section{Relaxation-oscillations}

\subsection{Model simplification}
In order to facilitate the analysis of relaxation-oscillations present in (\ref{eqn:MainModel}) we make several simplifying and reasonable assumptions. First, since the nondimensional snow line height at the Arctic is usually small \cite{Fow13} we take the first order approximation with $\kappa = 0$. As we have mentioned above, this forces the lower branch of the $\lambda$-nullcline to vanish leaving the upper one as a increasing function of $\theta$. Perhaps, taking a static snow line is not completely physically feasible but treating it as a first approximation is justified when the variations of the temperature are small. Therefore, our subsequent considerations can be thought as meaningful only when we are investigating vicinity of the critical point representing our climate. The model is not suitable for describing large excursions from it. 

As was also shown in \cite{Plo18}, the ice sheet extent $\lambda$ is a small number. Specifically, we always have $\lambda < 2/9$ but in reality the upper bound is even smaller. In that case the snow line can be expanded into Taylor series for $\kappa = 0$
\begin{equation}
	\lambda_0(\lambda) = 1-4\lambda + O(\lambda^2) \quad \text{as} \quad \lambda\rightarrow 0^+.
\end{equation}
Moreover, from the definition of $\xi$ and the model data in Tab. \ref{tab:Symbols} we can see that its magnitude is small letting us to approximate even further
\begin{equation}
\label{eqn:Lambda0Simp}
	\left(1+\xi(\theta)\right)\lambda_0(\lambda) -1 = 1-4\lambda+\xi(\theta) -1 + O\left(\lambda^2\right) + O\left(\xi\lambda^2\right) = \xi(\theta) - 4\lambda + O\left(\lambda^2\right) + O\left(\xi\lambda^2\right),
\end{equation}
as $\lambda\rightarrow 0^+$ and $\xi \rightarrow 0^+$. Further, we take the simplest physically sensible form of the continental albedo
\begin{equation}
	\alpha_c(\lambda) = \alpha_0 + \alpha_1 \lambda.
\end{equation}
Although it is not a sigmoid function, the small variations in it allow us to use the linear approximation. This is consistent with the assumption that $\kappa = 0$. Now, we can write (\ref{eqn:MainModel}) as 
\begin{equation}
	\left\{
	\begin{array}{l}
	\dfrac{d\theta}{d\tau} = \mu \left(1-\beta-\gamma\alpha_0-\dfrac{1}{2}(1-\gamma)(\alpha_-+\alpha_+)-\gamma\alpha_1\lambda+\dfrac{1}{2}(1-\gamma)(\alpha_--\alpha_+)\sigma\left(\dfrac{\theta-\theta_\alpha}{\Delta\alpha}\right)-\theta\right), \vspace{6pt}\\
	\dfrac{d\lambda}{d\tau} = \dfrac{1}{2}(\xi_-+\xi_+)\sqrt{\lambda} \left(1+\dfrac{\xi_+-\xi_-}{\xi_-+\xi_+}\sigma\left(\dfrac{\theta-\theta_\xi}{\Delta\xi}\right)-\dfrac{8}{\xi_-+\xi_+}\lambda + O\left(\lambda^2\right)+O(\xi(\theta)\lambda^2)\right),
	\end{array}	
	\right.
\end{equation}
as $\lambda\rightarrow 0^+$ and $\xi \rightarrow 0^+$. If we neglect the higher order terms, define the following
\begin{equation}
\begin{split}
	a = 1-\beta-\gamma\alpha_0-\dfrac{1}{2}(1-\gamma)(\alpha_-+\alpha_+), \quad b = \frac{1}{8}\gamma\alpha_1 (\xi_-+\xi_+), \\
	c = \frac{1}{2}(1-\gamma)(\alpha_--\alpha_+), \quad d = \frac{\xi_+-\xi_-}{\xi_-+\xi_+}, \quad \nu =   \frac{1}{8}\sqrt{2(\xi_-+\xi_+)}\gamma\alpha_1\mu, \quad x_\alpha = \theta_\alpha, \quad x_\xi = \theta_\xi,
\end{split}
\end{equation}
and rescale the variables by
\begin{equation}
	x = \theta, \quad y = \frac{8}{\xi_-+\xi_+} \lambda, \quad t = \sqrt{2\left(\xi_-+\xi_+\right)}\tau,
\end{equation}
we obtain the following dynamical system
\begin{equation}
\label{eqn:MainSystem}
\left\{
\begin{array}{l}
	\dfrac{dx}{dt} = \nu\left(f(x)-y\right), \vspace{6pt}\\
	\dfrac{dy}{dt} = \sqrt{y}\left(g(x)-y\right),
\end{array}	
\right.
\end{equation}
with the nullclines
\begin{equation}
\label{eqn:Nullclines}
	f(x) = \frac{1}{b} \left(a+c\sigma\left(\dfrac{x-x_\alpha}{\Delta\alpha}\right)-x\right), \quad g(x) = 1 + d \sigma\left(\dfrac{x-x_\xi}{\Delta\xi}\right).
\end{equation}
Notice that now $x,y$ and $t$ are of order of unity. For simplicity and without any lose of generality we have taken the sigmoid functions $\sigma$ to be of the same family.

We have completed the derivation of the simplified model of climate under the assumption the the variations of the snow line are small enough for us to take $\kappa=0$. We can think of this description as a magnification of the vicinity of the present climate oscillations. 

\begin{figure}
	\centering
	\includegraphics[scale=0.9]{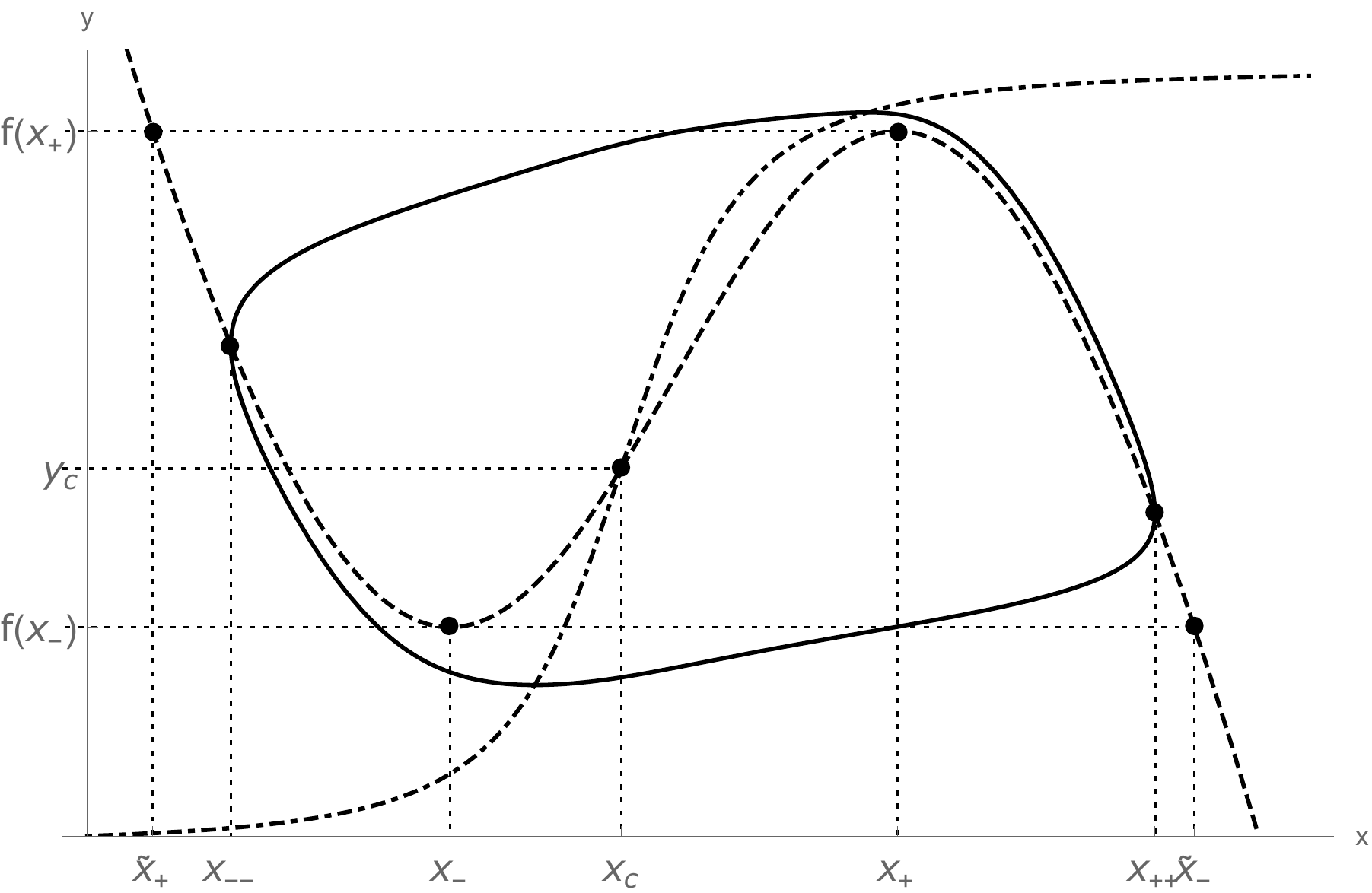}
	\caption{The phase plane of the system (\ref{eqn:MainSystem}). The solid line indicates the limit cycle; dashed line depicts the nullcline $f$; dot-dashed line is the graph of $g$. Various points are defined in (\ref{eqn:Nullclines}), (\ref{eqn:Extrema}), (\ref{eqn:ExtremaExpanded}), and (\ref{eqn:x++}). }
	\label{fig:PhasePlane}
\end{figure}

It is easy to note that $g$ is an increasing sigmoid function while $f$ has a typical S-shape having one minimum and one maximum (see Fig. \ref{fig:PhasePlane}) given by
\begin{equation}
\label{eqn:Extrema}
	x_\pm = x_\alpha \pm \Delta\alpha \left(\sigma'\right)^{-1}\left(\frac{\Delta\alpha}{c}\right) , \quad f(x_\pm) =\frac{1}{b}\left(a\pm c\sigma\left(\left(\sigma'\right)^{-1}\left(\frac{\Delta\alpha}{c}\right)\right)-x_\pm\right).
\end{equation}
It is also convenient to decompose $f$ into three branches
\begin{equation}
f = f_{s-} \cup f_u \cup f_{s+},
\end{equation}
where subscripts denote the stable and unstable parts of $f$. On their respective domains the branches are monotone and hence invertible. More precisely
\begin{equation}
	f_{s-} = f|_{(-\infty,x_-]}, \quad f_{u} = f|_{[x_-,x_+]}, \quad f_{s+} = f|_{(x_+,\infty)},
\end{equation} 
where the vertical bars indicate the restriction of the domain. Moreover, for future convenience we define the following points 
\begin{equation}
\label{eqn:ExtremaExpanded}
	\widetilde{x}_\pm = f_{s\mp}^{-1}\left(f(x_\pm)\right).
\end{equation} 

Before we move to the analysis we have to make several assumptions concerning the parameters involved. Specifically, in order for the relaxation-oscillations to occur there should exist a unique unstable critical point which we denote by $x_c$. It suffices then to make the following assumptions.
\begin{align}\label{eqn:Assumptions}
	f(x_-) > g(x_-) \quad \text{and} \quad f(x_+) < g(x_+) \quad (\text{only one critical point})&, \\
	g'(x_c) > f'(x_c) > 0 \quad (\text{the critical point is unstable})&, \\
	\nu > \nu_c \quad (\text{existence of the limit cycle}),
\end{align}
where the stability result can be proved essentially in the same way as in the previous section while the critical parameter is given by 
\begin{equation}
	\nu_c = \frac{3}{2} \frac{\sqrt{f(x_c)}}{f'(x_c)}.
\end{equation}
To get a quick peek on the reasonable size of the critical parameter we can make a crude estimate that a ice sheet oscillates between $70^\circ$ and $45^\circ$ of latitude when the temperature changes by at most $20^\circ$C (this is an overestimate). Then, assuming that $x_\pm$ bounds the amplitude of the oscillations and approximating $f'(x_c)$ by its finite difference we can estimate that $\nu_c \approx 0.15$. This can be considered as a small quantity. Conducting the exact calculations with our data it occurs that $\nu_c = 0.1$.  

Now, we can make some simple geometrical observations. From the Fig. \ref{fig:PhasePlane} we can see that for a physically reasonable parameter choice the function $g$ is almost constant for $x<x_-$ and $x > x_+$. Moreover, near $x_\xi$ it is well-approximated by a linear function. The latter property holds also for $f$ near $x_\alpha$. These observations can help us to find a reasonable approximation of the critical point.

\begin{prop}\label{prop:CriticalPoint}
Let (\ref{eqn:Assumptions}) be satisfied. Then 
\begin{equation}
	\left|x_c - \frac{g(x_\xi)-f(x_\xi)+g'(x_\xi)x_\xi-f'(x_\alpha)x_\alpha}{g'(x_\xi) - f'(x_\alpha)}\right| \leq C (x_+-x_-)^3,
\end{equation}
which for $x_\alpha=x_\xi=x_0$ can be written as
\begin{equation}
	\left|x_c - x_0 - \frac{g(x_\xi)-f(x_\xi)}{g'(x_\xi) - f'(x_\xi)}\right| \leq C (x_+-x_-)^3.
\end{equation}
where $C = \frac{1}{3} \max \left\{|f'''(x)|,|g'''(x)|:x\in[x_-,x_+]\right\}$.
\end{prop}
\begin{proof}
The proof goes by straightforward calculations. First, observe that if we expand $f$ and $g$ in Taylor series at the inflection points $x_\alpha$ and $x_\xi$ we have
\begin{align}
	f(x) &= f(x_\alpha) + f'(x_\alpha) (x-x_\alpha) + \frac{1}{6} f'''(\widetilde{x}_\alpha) (x-x_\alpha)^3, \\
	g(x) &= g(x_\xi) + g'(x_\xi) (x-x_\xi) + \frac{1}{6} g'''(\widetilde{x}_\xi) (x-x_\xi)^3.
\end{align}
Equalling both expansions we obtain
\begin{equation}
	x_c = \frac{1}{g'(x_\xi) - f'(x_\alpha)}\left(g(x_\xi)-f(x_\xi)+g'(x_\xi)x_\xi-f'(x_\alpha)x_\alpha\right) + R(x_c).
\end{equation}
For the remainder it holds
\begin{equation}
	\left|R(x_c)\right| \leq \frac{C}{2}\left(\left|x_c-x_\alpha\right|^3 + \left|x_c-x_\xi\right|^3\right) \leq C (x_+-x_-)^3,
\end{equation}
for the constant $C$ defined in the assertion. This finishes the proof.
\end{proof}
In practice the above approximation can be decently accurate and for our data yields an error of $3\times 10^{-4}$ which is smaller than necessary in this application. 

\subsection{Matched asymptotic analysis}

Next, we proceed to the main topic. By using the general theory \cite{Mis13,Kru01} it follows that the dynamical system (\ref{eqn:MainSystem}) possesses a stable limit cycle for all values of $\nu\geq \nu_c = 0.1$. Our general considerations from previous section indicate that the Hopf bifurcation occurs at $\nu_c$ and a cycle is born with a period of $200\times 10^3$ years which is twice the data-estimated value of the Pleistocene ice age oscillations. This indicates that, if the model is correct, $\nu$ should be larger than $\nu_c$. With the increase of $\nu$ the amplitude of oscillations increases until they become relaxation-oscillations. It means that the climate oscillates for almost the whole parameter range, that is the relaxation-oscillations are a robust phenomenon. Below we will use the matched asymptotic analysis to find their leading order behaviour. These results can be compared with the Van der Pol oscillator (see \cite{Kev13} for construction of the approximation and \cite{Mac90} for a rigorous justification). The following method borrows from \cite{Oma13} while in \cite{Nip88} an algorithimc approach has been developed with rigorous proofs of convergence. A simplified version that captures the quintessence of our calculations is described in \cite{Hol12}. 

We will approximate the solution of (\ref{eqn:MainSystem}) is several regions of the phase plane and the following material will be organized according to the particular layer. In order to track the progress it is helpful to consult Fig. \ref{fig:PhasePlane}. Moreover, as a notational convenience we will retain the same letter for denoting the exact solution with its leading order approximation. This is the only order to which we will match boundary layers and thus it will not cause any confusion. However, it will enhance the readability greatly. All the results of the following analysis are summarized in Theorem \ref{thm:Period} below.

First, introduce the parameter $\epsilon = \nu^{-1}$ and write (\ref{eqn:MainSystem}) as a singularly perturbed system
\begin{equation}
\label{eqn:MainSystemEpsilon}
\left\{
\begin{array}{rl}
	\epsilon\dfrac{dx}{dt} =& f(x)-y, \vspace{6pt}\\
	\dfrac{dy}{dt} =& \sqrt{y}\left(g(x)-y\right),
\end{array}	
\right.
\end{equation}
where we assume that $\epsilon\ll 1$. From the phase plane we see that there are two asymptotically stable branches of the slow manifold $y=f(x)$ (in the sense of Tikhonov-Levinson theory \cite{Oma13}). 

Suppose that we start our evolution on the left branch of the slow manifold where $g(x)<f(x)$ and we continue our oscillation sliding down-right. We will subsequently consider various time scales and the behaviour of $(x(t),y(t))$ on them.\\

\noindent\textbf{Outer layer (left)}. This is the slow phase of the oscillations. The leading order behaviour of (\ref{eqn:MainSystemEpsilon}) emerges when we set $\epsilon = 0$ and and then the equations constitute the reduced system
\begin{equation}
\left\{
\begin{array}{rl}
	f(x_o) =& y_o, \vspace{6pt}\\
	\dfrac{dy_o}{dt} =& \sqrt{y_o}\left(g(x_o)-y_o\right),
\end{array}	
\right.
\end{equation}
where the subscript denotes the outer approximation. This is a differential-algebraic system saying that the evolution takes place on the slow manifold $y=f(x)$. By differentiating the first equation and using the second we arrive at
\begin{equation}
\label{eqn:OuterLeadingOrder}
	f'(x_o) \frac{dx_o}{dt} = -\sqrt{f(x_o)}\left(f(x_o)-g(x_o)\right).
\end{equation}
Notice that this equation has a singularity at $x_o = x_-$ which has to be resolved. To this end, let $x_o(t) = x_- - \varphi(t)$ where $0<\varphi(t)\ll 1$. Then
\begin{equation}
	f''(x_-) \varphi \frac{d\varphi}{dt} = -\sqrt{f(x_-)}\left(f(x_-)-g(x_-)\right) + O(\varphi^2).
\end{equation} 
The leading order behaviour of $\varphi$ is thus given by
\begin{equation}
\label{eqn:OuterAsymRight}
	\varphi(t) \sim \sqrt{\frac{2}{f''(x_-)}\sqrt{f(x_-)} \left(f(x_-)-g(x_-)\right) (-t)} \quad \text{as} \quad t\rightarrow 0^{-}.
\end{equation}
Notice that we have incorporated the integration constant into the time variable setting $t=0$ precisely at the jumping point $x_0=x_-$. Since at that point the outer approximation becomes non-smooth, we have to find another set of equations describing the solution for $t>0$. \\

\noindent\textbf{Transition layer}. The singularity of the outer solution (\ref{eqn:OuterAsymRight}) suggests there should exist a stretching transformation that produces the distinguished limit. Hence, we introduce the transition layer variables
\begin{equation}
\label{eqn:TransitionStretching}
x = x_- + \epsilon^p x_t, \quad y = f(x_-) + \epsilon^q y_t, \quad t = \epsilon^p (t_t+\delta_t(\epsilon)),
\end{equation}
where $p,q,r$ have to be found. Plugging into (\ref{eqn:MainSystemEpsilon}) gives
\begin{equation}
\left\{
\begin{array}{rl}
\epsilon^{1+p-r}\dfrac{dx_t}{dt_t} =& f(x_-+\epsilon^p x_t)-f(x_-)-\epsilon^q y_t, \vspace{6pt}\\
\epsilon^{q-r}\dfrac{dy_t}{dt_t} =& -\sqrt{f(x_-)+\epsilon^q y_t}\left(f(x_-)-g(x_-+\epsilon^p x_t)+\epsilon^q y_t\right).
\end{array}	
\right.
\end{equation}
Now, we have to find the appropriate balance between various terms in the above. For the outer solution the derivative was assumed to be small but eventually blew-up, hence we expect that the correct balance will be between all three terms in the equation for $x_t$. Since $f(x_-+\epsilon^p x_t) - f(x_t) = f''(x_-) \epsilon^{2p} x_t^2/2 + O(\epsilon^{3p})$ and we have to keep the derivative of $y_t$ it follows that the only choice of parameters is $1+p-r=2p=q$ and $q-r=0$. Therefore
\begin{equation}
p=\frac{1}{3}, \quad q=r=\frac{2}{3}.
\end{equation}
Whence, the dynamical system under this stretching becomes
\begin{equation}
\left\{
\begin{array}{rl}
\dfrac{dx_t}{dt_t} =& \frac{1}{2}f''(x_-)x_t^2- y_t + O(\epsilon), \vspace{6pt}\\
\dfrac{dy_t}{dt_t} =& -\sqrt{f(x_-)}\left(f(x_-)-g(x_-)\right) + O(\epsilon^\frac{2}{3})
\end{array}	
\right.
\quad \text{as} \quad \epsilon\rightarrow 0^+.
\end{equation}
In the leading order the second equation integrates to a linear function and the first one becomes
\begin{equation}
\frac{dx_t}{dt_t} = \frac{1}{2}f''(x_-)x_t^2+\sqrt{f(x_-)}\left(f(x_-)-g(x_-)\right) t_t,
\end{equation}
where the integration constant has been incorporated into $\delta_t(\epsilon)$. This nonlinear equation is of Riccati type and can be solved with a substitution
\begin{equation}
x_t = -\frac{2}{f''(x_-)} \frac{u'}{u}.
\end{equation}
In this new variable the differential equation has the form	
\begin{equation}
\frac{d^2u}{dt_t^2} + \frac{1}{2}f''(x_-)\sqrt{f(x_-)}\left(f(x_-)-g(x_-)\right) t_t \, u= 0,
\end{equation}
which is Airy's equation having a solution
\begin{equation}
u(t_t) = E_t \,\text{Ai} \left(-C_t t_t\right) + F_t \,\text{Bi} \left(-C_t t_t\right) \quad \text{where} \quad C_t := \sqrt[3]{\frac{1}{2}f''(x_-)\sqrt{f(x_-)}\left(f(x_-)-g(x_-)\right)},
\end{equation}
with $\text{Ai}$ and $\text{Bi}$ being Airy's functions. From the known asymptotic behaviour of these (see \cite{AS}) we can see that the only choice for $x_t$ to match $x_i$ for $t_t\rightarrow -\infty$ is to choose $F_t=0$. In that case we have
\begin{equation}
\label{eqn:TransitionAsymLeft}
x_t(t_t) = \frac{2 C_t}{f''(x_-)} \frac{\text{Ai}'(-C_t t_t)}{\text{Ai}(-C_t t_t)} \sim \sqrt{\frac{2}{f''(x_-)}\sqrt{f(x_-)}\left(f(x_-)-g(x_-)\right) (-t_t)} \quad \text{as} \quad t_t \rightarrow -\infty.
\end{equation}
If we had retained $F_t$ an unnecessary minus sign would appear in front of the above asymptotic expansion. In order to conduct the match with the outer approximation we write 
\begin{equation}
x(t) = x_- + \epsilon^\frac{1}{3} x_t(\epsilon^{-\frac{2}{3}}t + \delta_t(\epsilon)) \sim \sqrt{\frac{2}{f''(x_-)}\sqrt{f(x_-)}\left(f(x_-)-g(x_-)\right) (-t+\epsilon^\frac{2}{3}\delta_t(\epsilon))} \quad \text{as} \quad \epsilon\rightarrow 0^+,
\end{equation}
and $|t|\ll \epsilon^{-2/3}$. In order to satisfy (\ref{eqn:OuterAsymRight}) we take $\delta_t(\epsilon) = 0$. 

The matching with outer layer is complete and we proceed to investigate behaviour for large $t_t$ where a matching with the inner layer is anticipated. A potential difficulty arises since (\ref{eqn:TransitionAsymLeft}) has a singularity at the first zero of Airy's function, that is to say at 
\begin{equation}
\label{eqn:AiryZero}
t_t^* = \frac{\zeta}{C_t} \quad \text{where } \zeta \text{ is smallest positive number such that Ai}(-\zeta) = 0. 
\end{equation}
Moreover, from the properties of Ai we know that the aforementioned zero is simple and
\begin{equation}
\label{eqn:TransitionAsymRight}
x_t(t_t) \sim \frac{2}{f''(x_-)} \frac{1}{t_t^* - t_t} - \frac{1}{3}\sqrt{f(x_-)}\left(f(x_-)-g(x_-)\right)t_t^*\left(t_t^*-t_t\right)\quad \text{as} \quad t_t \rightarrow \left(t_t^*\right)^-.
\end{equation}
Now, we will see how does this behaviour agree with the next layer.\\

\noindent\textbf{Inner layer}. Naturally, we proceed to investigate the inner layer approximation. To this end introduce the fast time scale $t_i = \epsilon^{-1} \left(t-\delta_i(\epsilon)\right)$ where $\delta_i(\epsilon)$ is a yet unknown translation indicating the beginning of the layer. This leads to
\begin{equation}
\left\{
	\begin{array}{rl}
	\dfrac{dx_i}{dt_i} &= f(x_i) - y_i,\\
	\dfrac{dy_i}{dt} &= \epsilon\sqrt{y_i}\left(g(x_i)-y_i\right),
	\end{array}	
\right.
\end{equation}
where now subscript denotes the inner approximation. The leading order behaviour emerges when we put $\epsilon = 0$ leaving only the first equation
\begin{equation}
	\frac{dx_i}{dt_i} = f(x_i) - Y(\epsilon),
\end{equation}
where $Y(\epsilon)$ is the integration constant that may depend on $\epsilon$ with an order smaller than one. Now, observe that if the above equation were supposed to yield a \textit{leading order} approximation to the solution of (\ref{eqn:MainSystemEpsilon}) we should have $x_i \rightarrow x_-$ as $t\rightarrow -\infty$ for $\epsilon = 0$. Therefore we put $Y = f(x_-) - \epsilon^\frac{2}{3} W$ and obtain
\begin{equation}
\label{eqn:InnerEq}
	\frac{dx_i}{dt_i} = f(x_i) - f(x_-) + \epsilon^\frac{2}{3}W.
\end{equation}
Notice that we anticipated the first meaningful power of $\epsilon$ in the asymptotic expansion. This is due to the fact that the $2/3$ exponent emerges in the transition layer approximation to which inner has to match for $t_i\rightarrow -\infty$. 

Next, focus only on the leading order case for which $\epsilon=0$. The above equation has two critical points: $x_-$ and $\widetilde{x}_-$ (see Fig. \ref{fig:PhasePlane}). The approach to the latter is exponential while to the former only algebraic since $f'(x_-)=0$. More specifically, setting $x_i = x_- + \varphi$ with $0<\varphi\ll 1$ lets us determine that
\begin{equation}
\label{eqn:InnerAsymLeft}
	\varphi(t_i) \sim -\frac{2}{f''(x_-)} \frac{1}{t_i} \quad \text{as} \quad t_i \rightarrow -\infty.
\end{equation}
In (\ref{eqn:InnerAsymLeft}) the integration constant has been incorporated into $\delta_i(\epsilon)$. 


In order to match inner and transition layers we use (\ref{eqn:InnerAsymLeft}) with $t_i$ expressed in terms of $t_t$, that is to say $t_i = \epsilon^{-\frac{1}{3}}(t_t - \epsilon^{-\frac{2}{3}}\delta_i(\epsilon))$. And hence
\begin{equation}
x_i(t_t) \sim x_- - \frac{2}{f''(x_-)} \frac{\epsilon^\frac{1}{3}}{t_t-\epsilon^{-\frac{2}{3}}\delta_i(\epsilon)} \quad \text{as} \quad \epsilon \rightarrow 0^+ \quad \text{and} \quad |t_t|\ll \epsilon^\frac{1}{3}.
\end{equation}
Remembering the stretching transformation (\ref{eqn:TransitionStretching}) we can successfully match $x_t$ and $x_i$ if $\delta_i(\epsilon) = \epsilon^{2/3} t_t^*$ which fixes the time shift. The matching with the higher order requires retaining $\epsilon$ in (\ref{eqn:InnerEq}) and plugging $x_i(t_i) = x_-+\widetilde{\varphi}(t_i)$ with $0<\widetilde{\varphi}\ll 1$. This yields
\begin{equation}
	\frac{d\widetilde{\varphi}}{dt_i} = \frac{1}{2}f''(x_-) \widetilde{\varphi}^2 + \epsilon^\frac{2}{3} W + O\left(\widetilde{\varphi}^3\right) \quad \text{as} \quad \widetilde{\varphi} \rightarrow 0.
\end{equation}
This equation can be solved to give
\begin{equation}
	\widetilde{\varphi}(t_i) = \epsilon^\frac{1}{3}\sqrt{\frac{2W}{f''(x_-)}}\tan\left(\epsilon^\frac{1}{3}\sqrt{\frac{1}{2}f''(x_-)W} \, (t_i+V)\right),
\end{equation}
where $V$ is a time shift constant. The asymptotic expansion of the above requires some care. We are interested in letting $t_i\rightarrow -\infty$, i.e. when the solution leaves the inner layer to the left. On the other hand, when $\epsilon$ goes to zero the left asymptote of $\tan$ function approaches minus infinity. This suggests that we should expand as follows
\begin{equation}
\begin{split}
	x_i(t_i) &= x_- - \frac{2}{f''(x_-)} \frac{1}{t_i+V + \epsilon^{-\frac{1}{3}}\frac{\pi}{2 \sqrt{f''(x_-)W / 2}}} \\ 
	&+\frac{1}{3}\epsilon^\frac{1}{3} W \left(t_i+V + \epsilon^{-\frac{1}{3}}\frac{\pi}{2 \sqrt{\frac{1}{2}f''(x_-)W}}\right) + O\left(t_i+V + \epsilon^{-\frac{1}{3}}\frac{\pi}{2 \sqrt{f''(x_-)W / 2}}\right)^2,
\end{split}
\end{equation} 
as $t_i \rightarrow -V - \epsilon^{-1/3}\pi/\left(2 \sqrt{f''(x_-)W / 2}\right)$. Now, expressing the above in the transition layer time $t_t$ we have
\begin{equation}
\begin{split}
	x_i(t_t) &= x_- - \epsilon^\frac{1}{3} \frac{2}{f''(x_-)} \frac{1}{t_t-t_t^*} +\frac{1}{3}\epsilon^\frac{1}{3} W \left(t_t-t_t^*\right) + O(\epsilon^{2/3}),
\end{split}
\end{equation}
as $\epsilon\rightarrow 0$ and $|t_t|\ll \epsilon^{1/3}$, where we have conducted the irrelevant time shift with an appropriate choice of constant $V$ and used the previously determined value of $\delta_i(\epsilon)= \epsilon^{-2/3}t_t^*$. If we compare the above expansion with (\ref{eqn:TransitionAsymRight}) we see that they are compatible if we choose
\begin{equation}
\label{eqn:W}
	W = \sqrt{f(x_-)} \left(f(x_-)-g(x_-)\right) t_t^*.
\end{equation}
We have thus completed the matching of inner layer with the transition approximation. 

We move further and investigate the behaviour of $x_i$ for large $t_i$. Due to (\ref{eqn:InnerEq}) and (\ref{eqn:W}) we see that it approaches the following limit
\begin{equation}
\label{eqn:x++}
\begin{split}
	x_{++}(\epsilon) &= f_{s+}^{-1}\left(f(x_-)-\epsilon^{2/3}\sqrt{f(x_-)} \left(f(x_-)-g(x_-)\right) t_t^*\right), \\
	&= \widetilde{x}_- - \frac{\epsilon^{2/3}}{f'(\widetilde{x}_-)} \sqrt{f(x_-)} \left(f(x_-)-g(x_-)\right) t_t^* + o(\epsilon^{2/3}) \quad \text{as} \quad \epsilon\rightarrow 0.
\end{split}
\end{equation}
To determine the pace of the approach we set $x_i = x_{++} - \psi$ with $0<\psi\ll 1$ to have
\begin{equation}
\label{eqn:InnerAsymRight}
\psi(t_i) \sim C_i \exp\left({f'(x_{++}(\epsilon))t_i}\right) \quad \text{as} \quad t_i \rightarrow \infty,
\end{equation}
where the integration constant has been incorporated into $C_i$. Note that $f'(x_{++})<0$. Next, we have to check if this behaviour matches the right branch of the outer layer. \\

\noindent\textbf{Outer layer (right)}. It is convenient to introduce the following right outer variables
\begin{equation}
	x = x_{p}, \quad y = y_{p}, \quad t = t_p + \delta_p(\epsilon),
\end{equation} 
Similarly as before in (\ref{eqn:OuterAsymRight}) we can show that $x_p$ has has a singular derivative at $x_+$. Moreover, by setting $x_p = x_{++} + \varphi$ we can show that
\begin{equation}
\label{eqn:OuterRightAsymRight}
	\varphi(t_p) \sim \frac{\sqrt{f(x_{++}(\epsilon))}}{f'(x_{++}(\epsilon))}\left(g(x_{++}(\epsilon))-f(x_{++}(\epsilon))\right) t_p \quad \text{as} \quad t_p \rightarrow 0,
\end{equation}
where integration constant has been incorporated into $\delta_c(\epsilon)$. We can see that the above linear form cannot be matched to the exponential behaviour of the inner layer (\ref{eqn:InnerAsymRight}). Therefore, there exists a final layer that joins these two approximations.\\

\noindent\textbf{Corner layer}. In this layer we introduce the following scalings
\begin{equation}
	x = x_{++}(\epsilon) + \epsilon^p x_c, \quad y = f(x_{++}(\epsilon)) + \epsilon^q y_c, \quad t = \epsilon^r t_c + \epsilon^\frac{2}{3}t_t^* + \delta_c(\epsilon),
\end{equation}
where we have explicitly included the previously determined time lag. Moreover, we can assume that $\delta_c(\epsilon) = o(\epsilon^{2/3})$ as $\epsilon\rightarrow 0$. The use of the same auxiliary constants $p, q$ and $r$ should not confuse the Reader with those used in the transition layer. The system (\ref{eqn:MainSystemEpsilon}) now becomes 
\begin{equation}
\left\{
\begin{array}{rl}
	\epsilon^{1+p-r}\dfrac{dx_c}{dt_c} =& f(x_{++}+\epsilon^p x_c)-f(x_{++}(\epsilon))-\epsilon^q y_c, \vspace{6pt}\\
	\epsilon^{q-r}\dfrac{dy_c}{dt_c} =& \sqrt{f(x_{++}(\epsilon))+\epsilon^q y_c}\left(g(x_{++}(\epsilon)+\epsilon^p x_c)-f(x_{++}(\epsilon))-\epsilon^q y_c\right).
\end{array}	
\right.
\end{equation}
Since now $f'(x_{++})\neq 0$ we can obtain the distinguished limit by requiring thet $1+p-r=p=q$ and $q-r=0$. Therefore, 
\begin{equation}
	p = q = r = 1,
\end{equation}
and the system (\ref{eqn:MainSystemEpsilon}) becomes
\begin{equation}
\left\{
\begin{array}{rl}
	\dfrac{dx_c}{dt_c} =& f'(\widetilde{x}_-) x_c - y_c + O(\epsilon), \vspace{6pt}\\
	\dfrac{dy_c}{dt_c} =& \sqrt{f(\widetilde{x}_-)}\left(g(\widetilde{x}_-)-f(\widetilde{x}_-)\right) + O(\epsilon)
\end{array}	
\right. \quad \text{as} \quad \epsilon\rightarrow 0. 
\end{equation}
To the leading order we have 
\begin{equation}
	\frac{dx_c}{dt_c} = f'(\widetilde{x}_-) x_c - \sqrt{f(\widetilde{x}_-)}\left(g(\widetilde{x}_-)-f(\widetilde{x}_-)\right) t_c,
\end{equation}
where the integration constant has been incorporated into $\delta_c(\epsilon)$. The above linear equation can be solved yielding
\begin{equation}
	x_c(t_c) = \frac{\sqrt{f(\widetilde{x}_-)}\left(g(\widetilde{x}_-)-f(\widetilde{x}_-)\right)}{f'(\widetilde{x}_-)^2}\left(1+f'(\widetilde{x}_-)t_c\right) + C_c \exp\left(f'(\widetilde{x}_-)t_c\right),
\end{equation}
where $C_c$ is integration constant. Since $f'(\widetilde{x}_-)<0$ we have the following asymptotic behaviours
\begin{equation}
\label{eqn:CornerAsymRight}
	x_c(t_c) \sim \frac{\sqrt{f(\widetilde{x}_-)}\left(g(\widetilde{x}_-)-f(\widetilde{x}_-)\right)}{f'(\widetilde{x}_-)}t_c \quad \text{as} \quad t_c \rightarrow \infty,
\end{equation}
and
\begin{equation}
	x_c(t_c) \sim C_c \exp\left(f'(\widetilde{x}_-)t_c\right) \quad \text{as} \quad t_c \rightarrow -\infty.
\end{equation}
Since $t_c = t_i -\epsilon^{-1}\delta_c(\epsilon)$ we thus have
\begin{equation}
\begin{split}
	x_c(t_i) &\sim C_c \exp\left(f'(\widetilde{x}_-)\left(t_i-\epsilon^{-1}\delta_c(\epsilon)\right)\right) \\
	&= C_c \exp(f'(\widetilde{x}_-)t_i) \exp\left(-f'(\widetilde{x}_-)\epsilon^{-1}\delta_c(\epsilon)\right) \quad \text{as} \quad \epsilon\rightarrow 0.
\end{split}
\end{equation}
Remembering that $x=\widetilde{x}_-+\epsilon \, x_c$ in order to match corner and inner layers, i.e. the above equation with (\ref{eqn:InnerAsymRight}), we have to take
\begin{equation}
	C_c = C_i, \quad \delta_c(\epsilon) = \frac{\epsilon\ln\epsilon}{f'(\widetilde{x}_-)}.
\end{equation}
We can see that $\delta_c(\epsilon) \ll \delta_i(\epsilon)$ and hence the solution spends less time in the corner layer than in the two previous ones.

On the other hand, for the matching inner with right outer layer we express the time scales as $t_c = \epsilon^{-1} (t_p+\delta_p(\epsilon)-\delta_c(\epsilon)-\epsilon^{2/3}t_t^*)$ and write (\ref{eqn:CornerAsymRight}) in the form
\begin{equation}
	x_c(t_p) \sim \epsilon^{-1} \frac{\sqrt{f(\widetilde{x}_-)}\left(g(\widetilde{x}_-)-f(\widetilde{x}_-)\right)}{f'(\widetilde{x}_-)} (t_p+\delta_p(\epsilon)-\delta_c(\epsilon)-\epsilon^{2/3}t_t^*) \quad \text{as} \quad \epsilon\rightarrow 0 \quad \text{and} \quad |t_p|\ll \epsilon.
\end{equation}
The consistency with (\ref{eqn:OuterRightAsymRight}) can be achieved by taking
\begin{equation}
	\delta_p(\epsilon) = \epsilon^{\frac{2}{3}} t_t^*,
\end{equation}
because $\delta_c(\epsilon)$ is of higher order. 

So far the solution have travelled from $x_-$ on the left branch of the stable manifold to $x_{++}$ on the right one. It left outer layer and went through transition, inner and corner layers in time $\epsilon^{2/3} t_t^* + o(\epsilon^{2/3})$ as $\epsilon\rightarrow 0$. The further journey takes the solution to the larger fold $x_+$ and then the evolution is mirror reflected with a change of $x_-$ to $x_+$ and $x_{++}$ to $x_{--}$. Therefore, in the last part of the proof we have to calculate the time spent on the right branch of the stable manifold. Denote it by $T_p$ and integrate (\ref{eqn:OuterLeadingOrder}) with $x_o$ replaced by $x_p$. Therefore,
\begin{equation}
	T_p = \int_{x_{++}}^{x_+} \frac{f'(x)dx}{\sqrt{f(x)}(g(x)-f(x))} = \int_{f(x_-)-\epsilon^{2/3}W}^{f(x_+)} \frac{dw}{\sqrt{w}(g(f_{s+}^{-1}(w))-w)},
\end{equation}
where $W$ is defined in (\ref{eqn:W}). Manipulating further we can separate the leading order behaviour
\begin{equation}
\begin{split}
	T_p &= I_+\left(f(x_-),f(x_+)\right) + I_+(f(x_-)-\epsilon^\frac{2}{3}W,f(x_-)) \\
	&= I_+\left(f(x_-),f(x_+)\right) + \epsilon^\frac{2}{3} t_t^* \frac{f(x_-)-g(x_-) }{g(\widetilde{x}_-)-f(x_-)} + o(\epsilon^\frac{2}{3}) \quad \text{as} \quad \epsilon\rightarrow 0.
\end{split}
\end{equation}
The solution has now arrived at $x_+$ from where it can return to the left stable branch of the manifold. The time necessary for this travel can be found essentially as before giving (\ref{eqn:Period}). The amplitudes of the limit cycle is clearly given by $x_{++}-x_{--}$ and $f(x_+)-f(x_-)$. Whence, we have proved the following theorem.

\begin{thm}
	\label{thm:Period}
	Let (\ref{eqn:Assumptions}) be satisfied. Then the period of the relaxation-oscillations has the following asymptotic expansion 
	\begin{equation}
	\label{eqn:Period}
	\begin{split}
	T &= I_+\left(f(x_-),f(x_+)\right) + I_-\left(f(x_+),f(x_-)\right) \\ 
	&+ \nu^{-\frac{2}{3}} \zeta \left(\frac{1}{\sqrt[3]{\frac{1}{2}f''(x_-)\sqrt{f(x_-)}\left(f(x_-)-g(x_-)\right)}}\left(1+\frac{f(x_-)-g(x_-)}{g(\widetilde{x}_-)-f(x_-)}\right)\right. \\
	&+\left.\frac{1}{\sqrt[3]{\frac{1}{2}f''(x_+)\sqrt{f(x_+)}\left(f(x_+)-g(x_+)\right)}}\left(1 + \frac{g(x_+) -f(x_+)}{f(x_+)-g(\widetilde{x}_+)}\right)\right)
	+ o(\nu^{-\frac{2}{3}}),
	\end{split}
	\end{equation}
	as $\nu\rightarrow\infty$, where
	\begin{equation}
	\label{eqn:I}
	I_\pm(u,v) = \int_{u}^{v} \frac{dw}{\sqrt{w}(g(f_{s\pm}^{-1}(w))-w)},
	\end{equation}
	and $-\zeta \approx -2.338$ is the largest zero of Airy's function Ai. Moreover, the $x-$ and $y-$amplitudes of the cycle are
	\begin{equation}
	\begin{split}
	A_x &= \widetilde{x}_--\widetilde{x}_+ \\
	&+ \nu^{-\frac{2}{3}}\zeta\left(\frac{1}{f'(\widetilde{x}_+)}\sqrt[3]{\frac{2f(x_+)\left(f(x_+)-g(x_+)\right)^2}{f''(x_+)}}-\frac{1}{f'(\widetilde{x}_-)} \sqrt[3]{\frac{2f(x_-)\left(f(x_-)-g(x_-)\right)^2}{f''(x_-)}}\right) \\
	&+ o(\nu^{-\frac{2}{3}})
	\end{split}
	\end{equation}
	and
	\begin{equation}
	A_y = f(x_+) - f(x_-) + O(\nu^{-1}),
	\end{equation}
	as $\nu\rightarrow\infty$. 
\end{thm}

From the above calculations we can immediately see that it can be conducted to any arbitrary $f$ and $g$ satisfying the usual geometric conditions. The specific forms (\ref{eqn:Nullclines}) are irrelevant for all the considerations. Therefore, we have the following corollary. 
\begin{cor}
The assertions of Theorem \ref{thm:Period} hold for any twice-differentiable functions $f$ and $g$ where $f$ has a single non-degenerate minimum and maximum while $g$ is sigmoid. 
\end{cor}

Furthermore, it turns out that the period can be easily approximated by elementary functions.
\begin{cor}
Let (\ref{eqn:Assumptions}) be satisfied. Then, the period of relaxation-oscillations for sufficiently large $\nu$ can be bounded
\begin{equation}
\label{eqn:PeriodApprox}
	T_- \leq T_{ro} \leq T_+,
\end{equation}
where
\begin{equation}
\begin{split}
	T_-&=\phi\left(f(x_-),f(x_+),1-d\right)+\phi\left(f(x_+),f(x_-),1+d\right), \\
	T_+ &= \phi\left(f(x_-),f(x_+),g(x_-)\right)+\phi\left(f(x_+),f(x_-),g(x_+)\right),
\end{split}
\end{equation}
with
\begin{equation}
	\phi(u,v,w) = \frac{1}{\sqrt{w}}\ln \left(\frac{\sqrt{v}-\sqrt{w}}{\sqrt{u}-\sqrt{w}}\frac{\sqrt{u}+\sqrt{w}}{\sqrt{v}+\sqrt{w}}\right).
\end{equation}
Moreover, for sufficiently large $\nu$ the amplitude of $x=x(t)$ satisfies
\begin{equation}
\label{eqn:AmplitudeApprox}
	A_x \leq 2c + b \left(f(x_+)-f(x_-)\right).
\end{equation}
\end{cor}
\begin{proof}
We will find lower and upper estimates on the leading order term in the asymptotic formula for the period (\ref{eqn:Period}). Notice that in (\ref{eqn:Period}) the leading order term is written in terms of the integrals $I_\pm$ given in (\ref{eqn:I}). Since $g$ is bounded due to the properties of sigmoid functions, the term $g \circ f_{s\pm}^{-1}$ can be bounded from below and above by $1\pm d$. This simplifies the integral yielding
\begin{equation}
	\phi(v,u,1+d)=\int_u^v\frac{dw}{\sqrt{w}\left(1+d-w\right)} < I_\pm(u,v) < \int_u^v \frac{dw}{\sqrt{w}\left(1-d-w\right)} = \phi(v,u,1-d),
\end{equation}
which follows by simple calculation. Therefore, since $g$ is increasing and $f_{s+}^{-1}(w) \geq x_+$ for $w \geq f(x_-)$ we have
\begin{equation}
	\phi\left(f(x_+),f(x_-),1+d\right) < I_+\left(f(x_-),f(x_+)\right) \leq \phi\left(f(x_+),f(x_-),g(x_+)\right),
\end{equation}
and similarly
\begin{equation}
\phi\left(f(x_+),f(x_-),1-d\right) < I_-\left(f(x_+),f(x_-)\right) \leq \phi\left(f(x_+),f(x_-),g(x_-)\right).
\end{equation}
Adding both above inequalities yields the assertion. 

For the estimate on the amplitude observe that $f$ approaches its asymptotes from within, that is to say
\begin{equation}
	\frac{1}{b}\left(a-c-x\right) < f(x) < \frac{1}{b}\left(a+c-x\right),
\end{equation}
by the properties of sigmoid functions. Therefore, from (\ref{eqn:x++}) it follows that $x_{++} \leq f_{x_+}^{-1}(f(x_-)) < \widehat{x}_{++}$ where $\widehat{x}_{++}$ is the abscissa of the point of intersection of the asymptote with a value $f(x_-)$. Completely analogous reasoning concerning the left branch of $f$ implies the assertion. This ends the proof.
\end{proof}

The above corollary gives us some useful formulas for determining the period od the oscillations. Notice that in order to find the lower bound we need only four numbers: two local extrema of $f$ and two bounds of $\xi$. By Theorem \ref{thm:Period} we know that the leading order term of $y-$amplitude of the oscillations is equal to the difference between these two extrema. Whence, knowing the ice sheet extent and its physical properties we are able to estimate the period od ice ages. Note that this result assumes only two essential mechanisms: ice-albedo and temperature-precipitation feedbacks. It is a \emph{universal} formula independent of any parametrization! For our data we have
\begin{equation}
	T_- = 95.8 \times 10^3 \; \text{years}, \quad T_+ = 125 \times 10^3 \; \text{years},
\end{equation}
which gives a very reasonable estimate of the oscillations during the last one million years where period is approximately $100\times 10^3$ years. Of course, by manipulating the parameters we can obtain other sensible outcomes of our model. However, the main point is that a simple two-dimensional model can provide a realistic leading-order approximation to the very complex natural system.

Notice that formulas for both period estimates in (\ref{eqn:PeriodApprox}) contain logarithmic terms. It is a simple observation that $T_+$ can diverge to infinity when $g(x_\pm)\rightarrow f(x_\pm)$. However, this approach is logarithmic hence very slow and, therefore, is in principle resistant to small variations in the model parameters. Since not every parameter of our model is directly measurable, this is a valuable feature of the formula for the period that allows for certain error. 

An exemplary illustration of the limit cycle of (\ref{eqn:MainSystem}) is given on Fig. \ref{fig:LimitCycle}. We can see that the range of the ice sheet extent shows realistic values while the variation of the temperature is somewhat too large. This is probably due to the initial simplification of an immobile snow line, i.e. setting $\kappa = 0$ (which will be removed in our future work). However, the period lies very close to the real-world value: for $\nu = 10$ it is equal to $121\times 10^3$~y. The most noteworthy observation is the presence of asymmetric relaxation-oscillations. This is one of the most characteristic features of the Pleistocene ice-ages. As we mentioned above, our model exhibits such a behaviour for all $\nu > \nu_c = 0.1$. 

\begin{figure}
	\centering
	\includegraphics[scale=1]{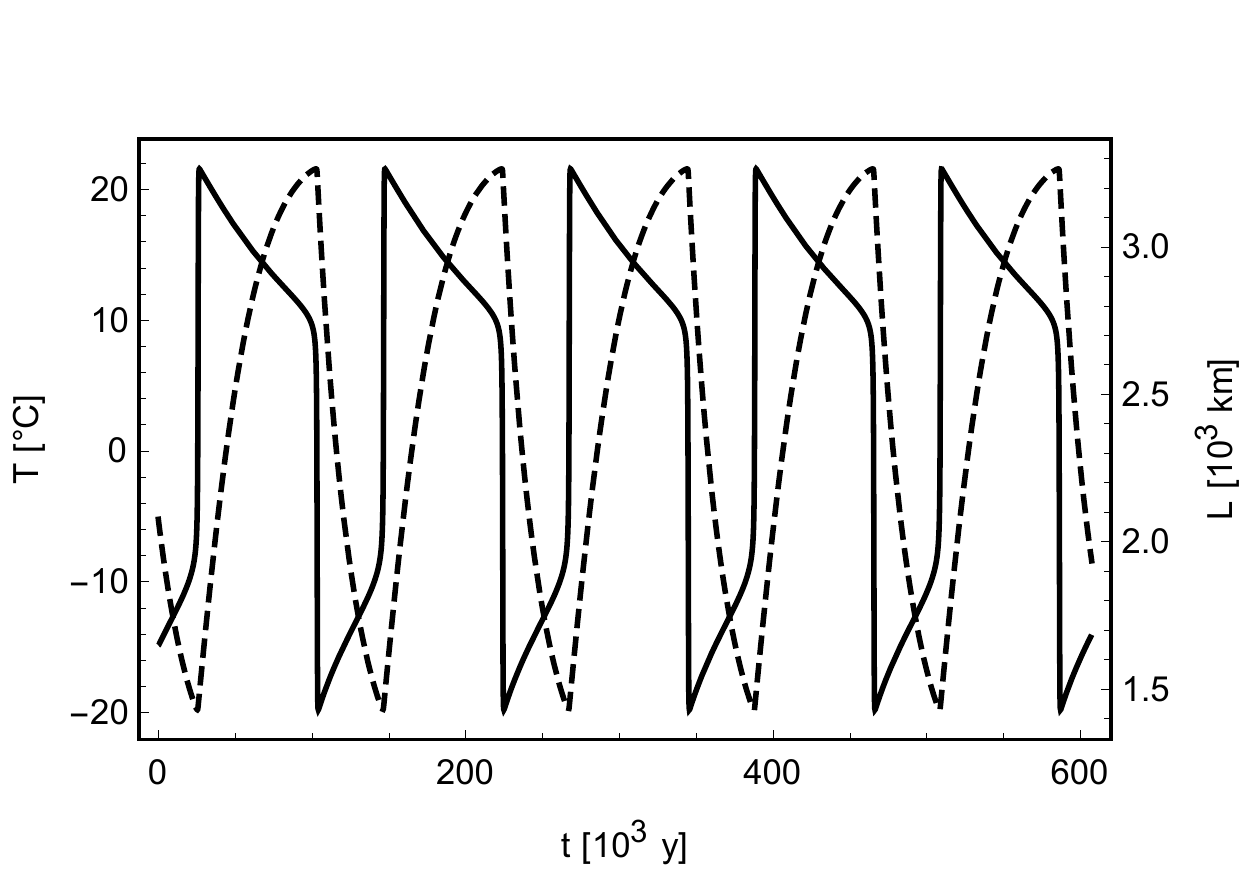}
	\includegraphics[scale=1]{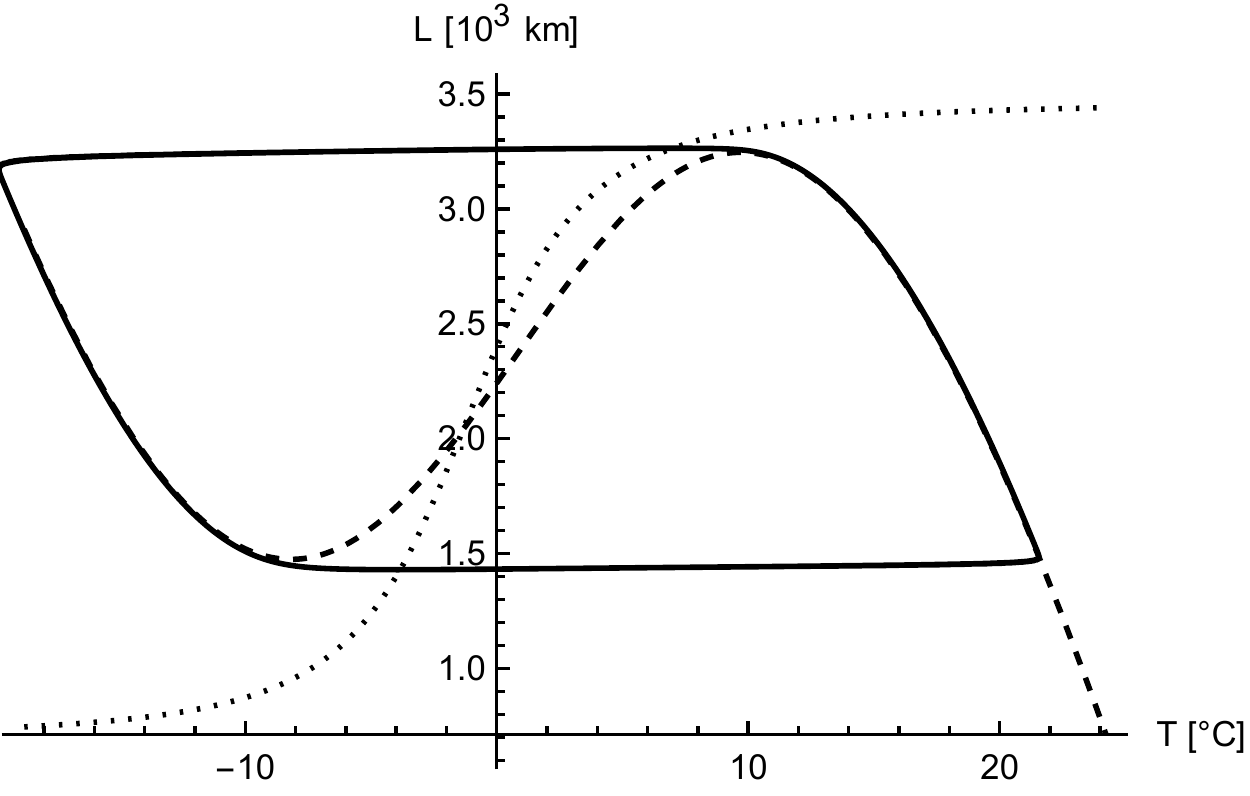}
	\caption{Time series (top) and the phase plane (bottom) of the dimensional form of the system (\ref{eqn:MainSystem}). On the top: solid line is the temperature while dashed line represents ice sheet extent. On the bottom: solid line is the orbit, dashed line is the $f$-nullcline while dotted line represents the $g$-nullcline. Here, $\nu = 10$.}
	\label{fig:LimitCycle} 
\end{figure}

The numerical verification of the asymptotic formula (\ref{eqn:Period}) is depicted on Fig. \ref{fig:Period}. We have calculated the period of the limit cycle for $\nu$ changing from $1$ to $10^7$. The results are presented on two plots: the first concerning the small values of $\nu$ and the second in log-log scale for the large values of the bifurcation parameter. For the latter we have plotted the difference between the numerical value and the leading order term of (\ref{eqn:Period}), denoted by $\overline{T}_{ro}$, versus the $\nu^{-2/3}$ term.  We can see that the accuracy is very good even for $\nu$ of order of unity. The power-law behaviour is clearly seen. Surprisingly, the formula is decently accurate even for small values of $\nu$. 

\begin{figure}
	\centering
	\includegraphics[scale=0.7]{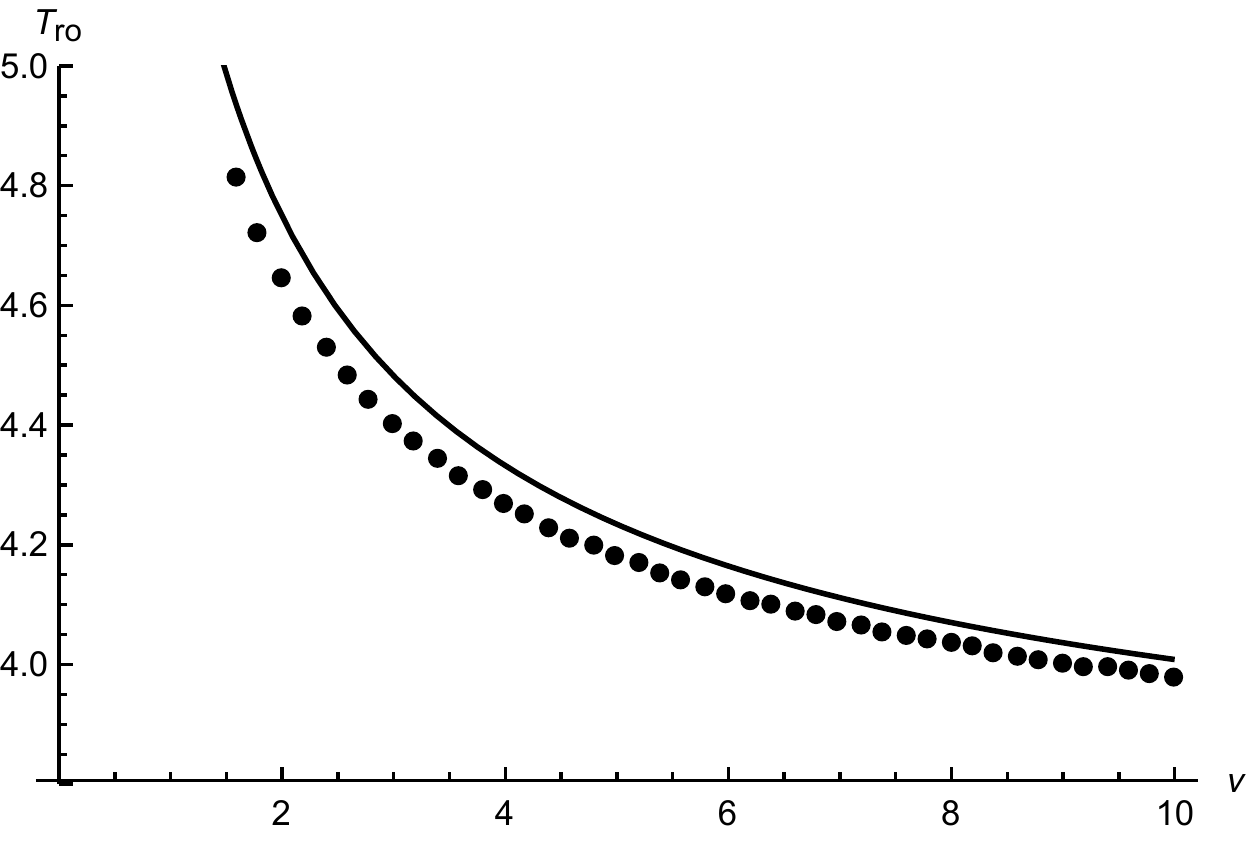}
	\includegraphics[scale=0.7]{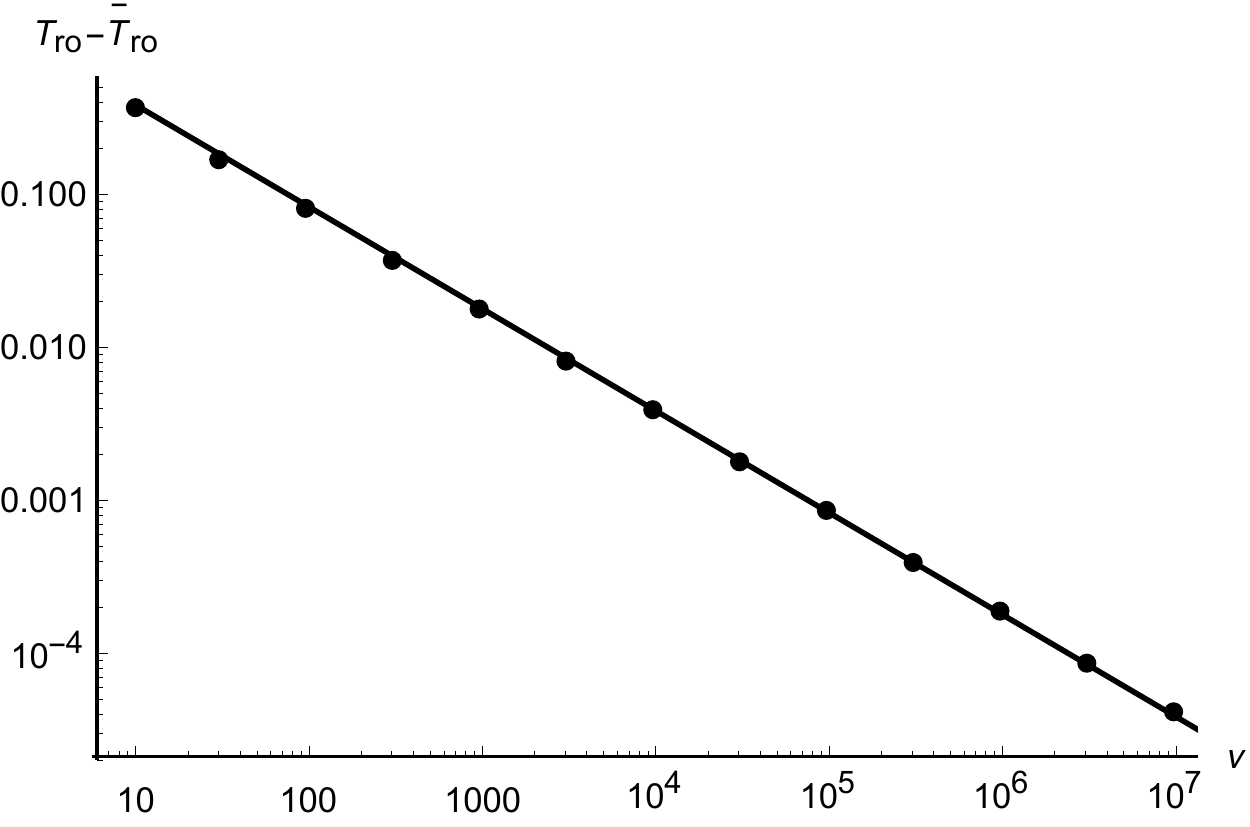}
	\caption{On the left: period of the relaxation-oscillations calculated numerically (points) and given by the formula (\ref{eqn:Period}) (solid line). On the right: the plot of the $\nu^{-2/3}$ term of (\ref{eqn:Period}) (solid line) and its numerical approximation (points). }
	\label{fig:Period}
\end{figure}

\section{Conclusion and future work}
There are several important remarks concerning our generalization of the KCG model. As we have seen, assuming two climate feedback mechanisms, i.e. ice-albedo and precipitation-temperature, leads to self-sustained oscillations after a Hopf bifurcation which happens for very small parameter magnitude. The astronomical forcing is not necessary to produce the dynamics of a realistic period. Moreover, the mathematical formulations of the involved feedbacks can be arbitrary as long as they satisfy natural assumptions of monotonicity and boundedness. This frees the model from unnecessary ad-hoc choice of the complex system parametrizations and furnishes a more robust model. 

In the vicinity of the Hopf-bifurcating critical point a stable limit cycle can exist provided there are no other critical points nearby. This let us to build an approximation of the periodic orbit of the model and analyse its relaxation-oscillations. Remarkably, a simple but physical model such as ours can exhibit a complex behaviour resembling the Pleistocene glaciations without any additional mechanism added. We have provided an asymptotic formula for the oscillation period and verified that it is decently accurate for a wide range of the bifurcation parameter. Although the exact closed-form formula for the leading order term in the period cannot be found, a simple estimates can be readily derived. The lower one is independent of the particular form of the parametrizations while the higher is insensitive to measurement errors. Both of them yield realistic estimates of the internal climate oscillations. 

In our future work we will try to incorporate several other degrees of freedom and analyse the corresponding dynamical system. These may include a slowly varying $\nu$ used to model the Middle Pleistocene Transition where the period of glaciations changed from $40$ to $100$ thousands of years. Another point of interest will be to include the viscoelastic response of the lithosphere under the load of the ice sheet. We also plan to investigate the role of the carbon cycle.  


\end{document}